\documentclass[12pt]{amsart}
\usepackage{amscd,amssymb,times,fullpage}
\newtheorem{theorem}{Theorem}[section]

\newtheorem{proposition}[theorem]{Proposition}
\newtheorem{corollary}[theorem]{Corollary}
\newtheorem{conjecture}[theorem]{Conjecture}
\theoremstyle{definition}
\newtheorem{definition}[theorem]{Definition}
\newtheorem{example}[theorem]{Example}

\theoremstyle{remark}
\newtheorem{remark}[theorem]{Remark}

\numberwithin{equation}{section}

\newcommand{\R}{\mathbb{ R}}

\newcommand{\RR}{\mathcal{ R}}

\newcommand{\DD}{\mathcal{ D}}

\def\bC{{\mathbb C}}

\def\bR{{\mathbb R}}
\def\bH{{\mathbb H}}

\begin{document}

\title{Contact pairs and locally conformally symplectic structures}
\author{G.~Bande}
\address{Dipartimento di Matematica e Informatica, Universit\`a degli Studi di Cagliari, Via Ospedale
72, 09124 Cagliari, Italy}
\email{gbande{\char'100}unica.it}
\author{D.~Kotschick}
\address{Mathematisches Institut, {\smaller LMU} M\"unchen,
Theresienstr.~39, 80333~M\"unchen, Germany}
\email{dieter{\char'100}member.ams.org}


\thanks{The first author was supported by the project \textit{Start-up giovani ricercatori} of the Universit\`a degli Studi di Cagliari. 
This work was begun while the second author was a Visiting Professor at the Universit\`a degli Studi di Cagliari in 2008}
\date{1 June 2010;  \copyright{\ G.~Bande and D.~Kotschick 2008--2010}}
\subjclass[2000]{Primary 53C25, 53C55, 57R17; Secondary 53C12, 53C15, 58A17}



\begin{abstract}
We discuss a correspondence between certain contact pairs on the one hand, and certain locally conformally symplectic forms on the other. 
In particular, we characterize these structures through suspensions of contactomorphisms. If the contact pair is endowed with a normal metric, then the 
corresponding lcs form is locally conformally K\"ahler, and, in fact, Vaisman. This leads to classification results for normal 
metric contact pairs. In complex dimension two we obtain a new proof of Belgun's classification of Vaisman manifolds under
the additional assumption that the Kodaira dimension is non-negative. We also  produce many examples of manifolds admitting 
locally conformally symplectic structures but no locally conformally K\"ahler ones.
\end{abstract}

\maketitle

\section{Introduction}

The notion of a {\it contact pair} was introduced in~\cite{Bande1,BH1}, but later turned out to be the same as the so-called bicontact 
structures considered long ago by Blair, Ludden and Yano~\cite{BLY} in the context of Hermitian geometry. A contact pair of type 
$(h,k)$ on a $(2h+2k+2)$-dimensional manifold is a pair of one-forms $(\alpha, \beta)$, such that 
$\alpha\land(d\alpha)^{h}\land\beta\land (d\beta)^{k}$ is a volume form, $(d\alpha)^{h+1}=0$ and $(d\beta)^{k+1}=0$. To such a pair 
are associated two Reeb vector fields $A$ and $B$, uniquely determined by the following  conditions: $\alpha(A)=\beta(B)=1$,
$\alpha(B)=\beta(A)=0$ and $i_{A}d\alpha=i_{A}d\beta=i_{B}d\alpha=i_{B}d\beta=0$.

{\it Locally conformally symplectic} or lcs forms were introduced by Lee~\cite{Lee} and Vaisman~\cite{V0,lcs}. They are non-degenerate 
two-forms $\omega$ for which there exists a closed one-form $\theta$, called the Lee form, satisfying $d\omega=\omega \land \theta$.

It turns out that a contact pair $(\alpha, \beta)$ of type $(h,0)$ gives rise to the lcs form $d\alpha+\alpha \land \beta$. In this paper we provide a necessary 
and sufficient condition for an lcs form to arise from a contact pair in this way. More generally, we show that a {\it generalized contact pair} 
of type $(h,0)$, which is a particular {\it contact-contact structure} in the sense of \cite{BGK}, gives rise to a lcs form. We prove that closed 
manifolds carrying a generalized contact pair of type $(h,0)$ are completely characterized by the fact that they fiber over the circle with fiber 
a contact manifold and the monodromy acting by a contactomorphism.

We also consider normal metric contact pairs of type $(h,0)$ introduced in~\cite{BH2,BH3}. These are contact pairs endowed with two complex 
structures $J$ and $T$ which coincide on the intersection of the kernels of the one-forms $\alpha$, $\beta$ and such that $JA=B=-TA$, together 
with a metric $g$ compatible with both complex structures, whose fundamental forms are $d\alpha \pm \alpha \land \beta$. In particular $J$ and $T$ give opposite orientations. 

We prove the equivalence between normal metric contact pairs and non-K\"ahler Vaisman structures and we give several applications related to lcK geometry. 
In particular we give examples of manifolds carrying an lcs form but no lcK structure and we classify compact complex surfaces with non-negative Kodaira dimension 
carrying a Vaisman structure. The classification of Vaisman complex surfaces (for any Kodaira dimension) has been obtained by Belgun~\cite{Bel} using different methods. 
Our proof relies on older results from~\cite{Ko} about complex surfaces admitting a complex structure for both orientations, and the results of Wall~\cite{W1,W2} on complex 
surfaces admitting a geometry in the sense of Thurston.

\section{Definitions and background}\label{s:back}

\subsection{Contact pairs}\label{ss:cp}

Contact pairs were considered in~\cite{Bande1,BH1}. We refer the reader to those papers and to~\cite{BGK} for the basic properties.
Here we only recall the definition of the structure, and of the associated Reeb vector fields.

\begin{definition}[\cite{Bande1,BH1}]\label{d:cpair}
A pair $(\alpha, \beta)$ of $1$-forms on a manifold is said to be a contact pair 
of type $(h,k)$ if the following conditions are satisfied: $\alpha\land 
(d\alpha)^{h}\land\beta\land (d\beta)^{k}$ is a volume form, $(d\alpha)^{h+1}=0$
and $(d\beta)^{k+1}=0$.
\end{definition}
The forms $\alpha$ and $\beta$ have constant class $2h+1$ and $2k+1$ respectively, 
and the leaves of their characteristic foliations have induced contact structures.

\begin{proposition}[\cite{Bande1,BH1}]\label{p:Reeb}
For a contact pair $(\alpha,\beta)$ there exist two commuting vector fields $A$, 
$B$ uniquely determined by the following  conditions: $\alpha(A)=\beta(B)=1$,
$\alpha(B)=\beta(A)=0$ and $i_{A}d\alpha=i_{A}d\beta=i_{B}d\alpha=i_{B}d\beta=0$.
\end{proposition}

In this paper we will only consider contact pairs of type $(h,0)$, so that $\beta$ is a closed one-form.
The dimension of the manifold is then $2h+2$.

More generally, we will consider pairs $(\alpha,\beta)$ of one-forms such that $d\beta=0$ and
$\alpha\land (d\alpha)^{h}\land\beta$ is a volume form, without requiring that $\alpha$ have constant 
class. We shall refer to these pairs as {\it generalized contact pairs} (of type $(h,0)$). In this 
case the kernels of $\alpha$ and $\beta$ form a special kind of contact-contact structure in
the sense of~\cite{BGK}, and we shall freely use the basic results from that paper.

For a generalized contact pair one defines a Reeb distribution as follows:
\begin{definition}[\cite{BGK}]
    The Reeb distribution $\RR$ consists of the tangent vectors $Y$ 
    satisfying the equation $(i_Yd\alpha )\vert_{\ker(\alpha)\cap\ker(\beta)}=0$.
\end{definition}
It is easy to see that this is a smooth distribution of rank two. We can unravel 
the definition as follows. At every point the $2$-form $d\alpha$ has rank either $2h$
or $2h+2$. If its rank at a point is $2h$, then at that point the fiber of the Reeb 
distribution $\RR$ is the kernel of $d\alpha$. If the rank of 
$d\alpha$ at a point is $2h+2$, then the form is symplectic 
in an open neighbourhood of that point, and on that 
neighbourhood $\RR$ is the symplectic orthogonal of $\ker(\alpha)\cap\ker(\beta)$.  
\begin{definition}[\cite{BGK}]\label{reebdist}
The Reeb vector fields $A$, $B$ of $(\alpha, \beta )$ are the unique 
vector fields tangent to the Reeb distribution $\RR$ such that 
$\alpha (A)=\beta (B)=1$, $\alpha (B)=\beta (A)=0$.
\end{definition}
In the special case that $\alpha$ has constant class, this definition coincides with the one in Proposition~\ref{p:Reeb},
and in that case the two Reeb vector fields commute. In fact, Proposition~5.8 in~\cite{BGK} shows that for a generalized
contact pair the Reeb vector fields commute if and only if $d\alpha$ is of constant rank $2h$, which means that we have
a contact pair of type $(h,0)$ in the sense of the original Definition~\ref{d:cpair}.
The difference between these two situations is measured by the Reeb class of the characteristic foliation $\ker(\alpha\land (d\alpha)^h)$.

\subsection{Locally conformally symplectic forms}\label{ss:lcs}

The notion of locally conformally symplectic forms is due to Lee~\cite{Lee}, and, in more modern form,
to Vaisman~\cite{V0, lcs}. We refer the reader to those references and to~\cite{DO} for a more detailed
discussion. Here we just recall the definition and the most basic notions.

\begin{definition}\label{d:lcs}
A locally conformally symplectic or lcs form on a manifold $M$ is a non-degenerate two-form $\omega$ which 
can be rescaled locally, in a neighborhood of any point in $M$, so as to be symplectic.
\end{definition}
This condition is equivalent to requiring that 
\begin{equation}\label{eq:lcs}
d\omega = \omega\wedge\theta \ ,
\end{equation}
for some closed one-form $\theta$, called the {\it Lee form} of the lcs form $\omega$. 

We shall always assume that the dimension of $M$ is $\geq 4$, for 
otherwise $\omega$ is closed and a volume form. If $\dim M >4$, then the closedness of $\theta$ is 
automatic, as it follows from~\eqref{eq:lcs} by exterior differentiation and the observation that in these 
dimensions the wedge product with a non-degenerate form is injective.

The lcs property is preserved under conformal rescalings of $\omega$, and the Lee form of $e^f\omega$ is 
$\theta +df$. Thus the de Rham cohomology class of the Lee form is an invariant of a conformal class of lcs
forms, and vanishes if and only if the form is globally conformally symplectic.

Due to its non-degeneracy, every lcs form $\omega$ has associated to it a unique vector field $L$ on $M$ defined by the equation
\begin{equation}\label{eq:L}
i_L\omega = \theta \ .
\end{equation}
Clearly $L$ satisfies $\theta (L)=0$, and the flow of $L$ preserves both $\theta$ and $\omega$.

\subsection{Locally conformally K\"ahler and Vaisman manifolds}\label{ss:lcK}
Let us consider a Hermitian manifold $(M, J, g)$, where $g$ is a Riemannian metric and $J$ a complex structure compatible with $g$. 
Let $\omega$ be its fundamental two-form defined by $\omega(X,Y)=g(X, JY)$. 
\begin{definition}
The Hermitian manifold $(M, J, g)$ is said to be locally conformally K\"ahler (lcK) if its fundamental $2$-form $\omega$
is locally conformally symplectic (lcs).
\end{definition}
This property is preserved under conformal rescalings of the metric. 


We refer to~\cite{DO, ornea} for a detailed account of lcK manifolds. In what follows we will be concerned with a special class of lcK manifolds:
\begin{definition}
An lcK manifold $(M, J, g)$ is called Vaisman if its Lee form $\theta$ is parallel with respect to the Levi-Civita connection of $g$.
\end{definition}
Whenever $\theta$ is parallel, it has constant length $\vert\vert\theta\vert\vert$, and in the sequel we will assume that this is non-zero, 
for otherwise $\omega$ would be a K\"ahler form. Without loss of generality we may then assume $\vert\vert\theta\vert\vert =1$.

For any lcK manifold denote by $B$ the vector field which is dual to $\theta$ with respect to $g$ and $A=-JB$. 
Then $A$ and $B$ are called the Lee and anti-Lee vector fields. When $\theta$ is non-zero we also define $\beta=\vert\vert\theta\vert\vert^{-1}\theta$, 
and $\alpha=-J\beta=-\beta \circ J$. Let $U$ be the dual of $\beta$ with respect to $g$, and $V=-JU$. The following two propositions 
are reformulations of Propositions 4.2 and 4.3 of~\cite{DO}, taking into account our sign convention for the Lee form:
\begin{proposition}\label{p:charact vaisman}
Let $(M, J, g)$ be a lcK manifold with Lee form $\theta$. Then $M$ is Vaisman if and only if $||\theta||$ is constant and $U$ is Killing for $g$.
\end{proposition}
\begin{proposition}\label{p:properties-vaisman-mfd}
On a Vaisman manifold the following relations hold:
\begin{equation*}\label{eq:properties-vaisman-mfd}
\begin{split}
&L_UJ=L_VJ=0 \; , \quad\quad L_Vg=0 \ , \\
&[U,V]=0\; , \quad\quad d\alpha=\vert\vert\theta\vert\vert (\omega +\alpha \wedge \beta) \ .
\end{split} 
\end{equation*}
\end{proposition}
On any Vaisman manifold we denote by $\DD$ the rank $2$ distribution spanned by $U$ and $V$. This is integrable and 
$J$-invariant. Its $g$-orthogonal complement $\DD^{\perp}$ is not integrable, but is $J$-invariant, and equals $\ker (\alpha)\cap\ker(\beta)$.
We define a new almost complex structure $T$ compatible with $g$ by setting $T=J$ on $\DD^{\perp}$ and $T=-J$ on $\DD$.
The following proposition is straightforward, but does not seem to have been observed before.
\begin{proposition}\label{p:Tint}
On a Vaisman manifold $(M, J, g)$ the almost complex structure $T$ is integrable, and induces the orientation opposite to the 
one induced by $J$.
\end{proposition}
\begin{proof}
The statement about the orientations is clear since $T$ is defined by conjugating $J$ on a subbundle of odd complex rank.
By the Newlander--Nirenberg theorem, to check the integrability of $T$ it suffices to check the vanishing of its Nijenhuis tensor:
$$
N_T(X,Y)=2([TX,TY]-[X,Y]-T[TX,Y]-T[X,TY]) \ .
$$
Since this is a tensor, and is skew-symmetric, we only have to check the vanishing of $N_T(X,Y)$ in the following three cases:
both $X$ and $Y$ are in $\DD^{\perp}$, both $X$ and $Y$ are in $\DD$, or $X\in\DD$ and $Y\in\DD^{\perp}$. In the first two 
cases $N_T(X,Y)=N_J(X,Y)$ by the definition of $T$, and this vanishes by the integrability of $J$. For the final, third, case 
we may assume that $X$ is a constant linear combination of $U$ and $V$. Then $L_XJ=0$ by Proposition~\ref{p:properties-vaisman-mfd},
and so $[X,JY]=J[X,Y]$ and $[JX,JY]=J[JX,Y]$. Using this we calculate:
\begin{alignat*}{1}
N_T(X,Y) &=2([TX,TY]-[X,Y]-T[TX,Y]-T[X,TY])\\
&=2(-[JX,JY]-[X,Y]+T[JX,Y]-T[X,JY])\\
&=2(-J[JX,Y]-[X,Y]+T[JX,Y]-TJ[X,Y])\\
&=2((T-J)[JX,Y]+T(T-J)[X,Y]) \ .
\end{alignat*}
Since $X$ and $JX$ are constant linear combinations of $U$ and $V$ and $Y$ is in $\DD^{\perp}$, it follows from 
Proposition~\ref{p:properties-vaisman-mfd} that $[X,Y]$ and $[JX,Y]$ are also in $\DD^{\perp}$. Since $T-J$ vanishes
on $\DD^{\perp}$, we finally conclude the vanishing of $N_T(X,Y)$.
\end{proof}

\section{Contact pairs, lcs forms and fibrations over the circle}\label{s:rel}

In this section we explain the relation between locally conformally symplectic forms, (generalized)
contact pairs of type $(h,0)$, and suspensions of contactomorphisms.

Suppose that $\omega$ is an lcs form, and $X$ is a vector field satisfying $L_X\omega =0$. Then we have:
$$
\omega\wedge L_X\theta = \omega\wedge L_X\theta +L_X\omega\wedge\theta =L_X(\omega\wedge\theta) = L_Xd\omega = dL_X\omega = 0 \ ,
$$
so that, by the non-degeneracy of $\omega$, we conclude $L_X\theta = 0$, which, by the closedness of $\theta$, is equivalent to 
$d (\theta (X))=0$. Thus $\theta (X)$ is constant, and if it is non-zero we can normalize $X$ so that $\theta (X)= 1$.

The following result is a small elaboration on the work of Vaisman~\cite[Proposition 2.2]{lcs}:
\begin{proposition}\label{p:equiv}
Let $M$ be a smooth manifold of dimension $2h+2$. There is a bijection between the following two kinds of structures:
\begin{enumerate}
\item contact pairs $(\alpha,\beta)$ of type $(h,0)$, and
\item locally conformally symplectic forms $\omega$ with Lee form $\theta$ admitting a vector field $X$ satisfying $L_X\omega=0$ and $\theta (X)=1$.
\end{enumerate}
Under this bijection the Reeb vector fields $A$ and $B$ of $(\alpha,\beta)$ correspond to $L$ and $X$ respectively. Moreover,
$\omega^{h+1}$ and $\alpha\wedge (d\alpha)^h\wedge\beta$ define the same orientation on $M$.
\end{proposition}
\begin{proof}
Suppose we have a contact pair $(\alpha,\beta)$ of type $(h,0)$. Then $\omega = d\alpha +\alpha\wedge\beta$ is an lcs form with Lee form $\theta = \beta$.
Moreover, for the Reeb vector field $A$ we have $i_A\omega =\beta$ by the defining properties of the Reeb vector field. Thus $A=L$. For the other
Reeb vector field, $B$, we have $L_B\omega=0$ because the flow of $B$ preserves $\alpha$ and $\beta$ and therefore $\omega$, and $\beta (B)=1$
by definition. Thus $B$ has all the properties required of $X$.

Conversely, suppose $\omega$ is lcs with Lee form $\theta$, and $X$ satisfies $L_X\omega=0$ and $\theta (X)=1$. Then define $\alpha = -i_X\omega$
and $\beta = \theta$. We claim that this is a contact pair of type $(h,0)$ with Reeb vector fields $L$ and $X$. 

For dimension reasons $\omega^{h+1}\wedge\theta=0$. This implies
\begin{equation*}
\begin{split}
0 &=i_X(\omega^{h+1}\wedge\theta)= (h+1)i_X\omega\wedge\omega^h\wedge\theta + \theta(X)\omega^{h+1}\\
 &= \omega^{h+1}-(h+1)\omega^h\wedge\alpha\wedge\beta = (d\alpha)^{h+1} \ ,
\end{split}
\end{equation*}
where the last equality follows from
$$
d\alpha +\alpha\wedge\beta = -di_X\omega - i_X\omega\wedge\theta = -di_X\omega - i_Xd\omega +\theta(X)\omega = -L_X\omega +\omega = \omega \ .
$$
Thus the rank of $d\alpha$ at every point is at most $2h$. We also have 
$$
0  \neq \omega^{h+1} = (h+1)\alpha\wedge (d\alpha)^h\wedge\beta \ .
$$
As $\beta=\theta$ is closed, we conclude that $(\alpha,\beta)$ is indeed a contact pair of type $(h,0)$. 

To determine the Reeb vector fields note that 
\begin{alignat*}{1}
i_Ld\alpha &= i_Ld(-i_X\omega)=i_L(-L_X\omega+i_Xd\omega)=i_Li_X(\omega\wedge\theta)=i_L(i_X\omega\wedge\theta+\omega)\\
&=\omega(X,L)\theta+i_L\omega = -\omega(L,X)\theta+\theta = -(i_L\omega)(X)\theta+\theta=-\theta(X)\theta+\theta=0 \ ,
\end{alignat*}
and 
$$
i_Xd\alpha = i_Xd(-i_X\omega)=i_X(-L_X\omega+i_Xd\omega)=i_Xi_Xd\omega=0 \ ,
$$
where we have used the assumptions $L_X\omega=0$, $\theta (L)=0$ and $\theta(X)=1$ repeatedly.
Thus we have shown that $L$ and $X$ span the Reeb distribution of our contact pair. Checking how $\alpha$ evaluates on $L$ and $X$ and combining 
the result with $\theta (L)=0$ and $\theta(X)=1$, we see $A=L$ and $B=X$.

The two constructions we have given are clearly inverses of each other, and the claim about orientations follows from the above calculations.
Thus the proof is complete.
\end{proof}

There is a partial generalization of this result to the case of generalized contact pairs in place of contact pairs. 
\begin{proposition}
On a closed manifold every generalized contact pair $(\alpha,\beta)$ of type $(h,0)$ gives rise to lcs forms $\omega = d\alpha +c\alpha\wedge\beta$ for large enough
$c\in\R$. The Lee form $\theta$ of $\omega$ equals $c\beta$.
\end{proposition}
\begin{proof}
We have
$$
d\omega = d(d\alpha +c\alpha\wedge\beta)=c \ d\alpha\wedge\beta = \omega\wedge c\beta.
$$
To check non-degeneracy we compute
$$
\omega^{h+1} = (d\alpha +c\alpha\wedge\beta)^{h+1} = (d\alpha)^{h+1}+c(h+1)\alpha\wedge (d\alpha)^h\wedge\beta \ .
$$
For large enough $c$, the second summand, which is a volume form by the assumption on $(\alpha,\beta)$, dominates the first summand,
so that the right hand side is a volume form.
\end{proof}
In this case the equality $A=L$ no longer holds, in fact $L$ is in general not proportional to the Reeb vector field $A$.
The other Reeb vector field $B$ does not give an infinitesimal automorphism $X$ of the lcs form.

If a closed manifold $M$ admits a (possibly generalized) contact pair $(\alpha,\beta)$, then the existence of the closed non-vanishing 
one-form $\beta$ implies that $M$ fibers over the circle. By a perturbation of $\beta$ in the space of closed one-forms one can achieve that $\beta$ represents a rational
cohomology class, so that a primitive integral multiple of it defines a fibration $M\longrightarrow S^1$ with connected fibers.
As soon as the perturbation is small enough (in the $C^0$  norm), the new (integral) $\beta$ still forms a (generalized) contact pair together with the same $\alpha$ 
as before. In this case the Reeb vector field $B$ is the monodromy vector field of the fibration over $S^1$. The restriction of $\alpha$ to any
fiber $F$ is a contact form, and the monodromy preserves the contact structure. The mondromy preserves the contact form if and only if 
$d\alpha$ has rank $2h$, and not $2h+2$, everywhere, which means that we have a genuine contact pair of type $(h,0)$, and not a generalized one.


We summarize this discussion in the following:
\begin{proposition}
Every closed manifold admitting a generalized contact pair of type $(h,0)$ fibers over the circle with fiber a contact manifold and the monodromy acting by a contactomorphism. 
Conversely, every mapping torus of a contactomorphism admits a generalized contact pair of type $(h,0)$  and an induced lcs form.
\end{proposition}

\section{Normal metric contact pairs and Vaisman structures}\label{s:lck-characterization}

Metric and normal contact pairs have been studied in~\cite{BH2, BH3}. 
We begin this section by giving a reformulation of these notions. The reformulation is then used
to relate normal metric contact pairs to Vaisman manifolds and thereby obtain some classification results.

Let $(\alpha, \beta)$ be a contact pair of type $(h,0)$ on a manifold $M$, with Reeb vector fields $A$ and $B$. 
The tangent bundle $TM$ splits as $TM=G\oplus \bR A \oplus \bR B$, where $G$ is the subbundle $\ker \alpha \cap \ker \beta$. 
On $G$ the form $d\alpha$ is symplectic, and so $G$ can be endowed with an almost complex structure $J_0$ and a compatible metric $g_0$. 

A natural way to extend the almost complex structure to the whole tangent bundle is to require that it intertwines the Reeb vector fields. 
In this way we obtain two almost complex structures $J$ and $T$ on $TM$ which coincide on $G$ but are complex conjugates of each other 
on $\bR A \oplus \bR B$. In particular, they give opposite orientations. We call $J$ the almost complex structure for which $JA=B$ and $T$ the other one 
(with $TB=A$). 
Conversely, given an almost complex structure $J$ which preserves the splitting of $TM$ and satisfies $JA= B$,  one can form a unique $T$ as before
by conjugating $J$ on the Reeb distribution.

The Riemannian metric $g_0$ can be extended to the whole $TM$ by putting $g= g_0 \oplus\alpha^2 \oplus \beta ^2$, and this makes the splitting of $TM$ orthogonal. 
With this choice, the Reeb action of the contact pair becomes totally geodesic~\cite{BH2} and such a metric is called \textit{associated} to the contact pair. More precisely we have:
\begin{definition}\label{d:associated metric}
Let $M$ be a manifold endowed with a contact pair $(\alpha, \beta)$ of type $(h,0)$. Let $A,B$ be its Reeb vector fields and $J$ an almost complex structure such that $JA=B$. 
A Riemannian metric $g$ on $M$ is called {\it associated} to the contact pair if for all vector fields $X,Y$ we have:
\begin{equation*}
g(X, J Y)= (d\alpha-\alpha \wedge \beta) (X,Y) \, .\\
\end{equation*}
The $4$-tuple $(\alpha, \beta , J,g)$ will be called {\it metric contact pair}.
\end{definition}
The above discussion shows that, given the contact pair, a pair $(J,g)$ always exists. 
Also, observe that from Definition \ref{d:associated metric} one easily deduces that the Reeb vector fields $A,B$ are $g$-dual to $\alpha$ and $\beta$ respectively. Then the splitting $G\oplus \bR A \oplus \bR B$ is orthogonal with respect to $g$. Since $J$ is $g$-orthogonal and $JA= B$, it preserves $G$. This implies that $J=J_0 \oplus J_1$ where $J_0$ and $J_1$ are the almost complex structures induced by $J$ on $G$ and on $\bR A \oplus \bR B$ respectively. Then the almost complex structure $T$ defined by $J_0 \oplus (-J_1)$ is uniquely determined by $J$ and $g$.
It is clear that $T$ and $J$ coincide on $G$, are complex conjugate to each other on $\bR A \oplus \bR B$, and give opposite orientations.  Moreover we have $g(X, TY)= (d\alpha+\alpha \wedge \beta) (X,Y) $.

The triples $(M,J,g)$ and $(M,T,g)$ are almost Hermitian structures and their fundamental forms are $d\alpha- \alpha \wedge \beta$  and $d\alpha + \alpha \wedge \beta$ respectively.
\begin{definition}[\cite{BH2}]
The metric contact pair $(\alpha, \beta , J,g)$ is called \textit{normal} if both $J$ and $T$ are integrable.
\end{definition}
By Proposition 3.1 of \cite{BH3}, a metric contact pair $(\alpha, \beta , J,g)$ is normal if and only if $J$ is integrable and $L_A J=0$ or, equivalently, if $J$ is integrable and $L_B J=0$.

Now we want to make clear the link between metric contact pairs and lcK structures:
\begin{proposition}\label{p:contact-pair-lck}
Let $(\alpha, \beta , J,g)$ be a metric contact pair of type $(h,0)$ on a manifold $M$. If $J$ (resp. $T$) is integrable, then $(M,J, g)$ (resp. $(M,T, g)$) is lcK.
\end{proposition}
\begin{proof}
If $J$ is integrable, then $(M,J, g)$ is an Hermitian manifold and its fundamental form is $\omega=d\alpha-\alpha \wedge \beta$. 
By the proof of Proposition~\ref{p:equiv}, this is an lcs form with Lee form $-\beta$.
With the same argument we see that $(M,T, g)$ is lcK with Lee form $\beta$.
\end{proof}


Now we can characterize normal metric contact pairs in terms of Vaisman structures:
\begin{proposition}\label{p:equiv-normalcp-vaisman}
Let $M$ be a smooth manifold of dimension $2h+2$. There is a bijection (modulo constant rescaling of the metric) between the following two kinds of structures:
\begin{enumerate}
\item normal metric contact pairs $(\alpha, \beta , J,g)$ of type $(h,0)$, and
\item non-K\"ahler Vaisman structures $(J,g)$.
\end{enumerate}
\end{proposition}
\begin{proof}
If $(\alpha, \beta , J,g)$ is a normal metric contact pair on $M$, then $(M,J, g)$ is lcK by Proposition \ref{p:contact-pair-lck}. With respect to $g$, the Reeb vector fields $A$ and $B$ have length $1$ and are dual to $\alpha$ and $\beta$ respectively. So in particular the Lee form $\beta$ has constant length. Moreover, the normality condition implies that the Reeb vector fields are Killing for $g$ (see \cite[Theorem 4.1]{BH2} and~\cite[Proposition 3.1]{BH3}). 
We can then conclude that the lcK structure is Vaisman by applying Proposition \ref{p:charact vaisman}. 

Conversely, let us suppose that $(M,J, g)$ is a non-K\"ahler Vaisman manifold. Then its Lee form $\theta$ has non-zero constant length and after 
a constant rescaling of the metric $g$, we may assume that $\vert\vert\theta\vert\vert=1$. Then, by Proposition~\ref{p:properties-vaisman-mfd}, 
the fundamental form of $(M,J, g)$ becomes $\omega= d\alpha-\alpha \wedge \beta$. With the same notation as in Proposition~\ref{p:properties-vaisman-mfd}, we have $U=B$ and $V=-JB=A$, where $A$ and $B$ are the anti-Lee and the Lee vector field respectively. Since the flow of $B$ preserves $\beta$ and the complex structure $J$, it also preserves $\alpha=-J \beta$ and hence $\omega$. This in turns gives:
\begin{equation*}
L_B \omega=0 \, , \, \quad \quad \beta (B)=1\, .
\end{equation*}
Then the pair $(\alpha, \beta)$ is a contact pair by Proposition~\ref{p:equiv}, with Reeb vector fields $A$ and $B$. 
Moreover the complex structure $J$ intertwines $A$ and $B$, by the definition of the anti-Lee vector field.

Thus $(\alpha, \beta , J,g)$  is a metric contact pair. It is normal since the corresponding $T$ is also integrable 
by Proposition~\ref{p:Tint}.
\end{proof}

This correspondence allows us to apply results about complex manifolds, in particular the extensive work on Vaisman manifolds,
to study normal metric contact pairs of type $(h,0)$. We saw in Section~\ref{s:rel} that contact pairs of this type give rise
to fibrations over the circle with contact monodromy. For normal metric pairs the conclusion can be strengthened by
the structure theorem for Vaisman manifolds proved by Ornea and Verbitsky~\cite{OV1}, to the effect that the fibers of 
the fibration over the circle are Sasakian, and the monodromy is an automorphism of the Sasakian structure. In particular it
is an isometry of the Sasakian metric on the fibers, and so some power is isotopic to the identity, since the isometry group
of the fiber has finitely many components. Thus:
\begin{corollary}
Every closed manifold with a normal metric contact pair of type $(h,0)$ is finitely covered by the product of a Sasakian manifold
with $S^1$.
\end{corollary}
This corollary implies that the only manifolds supporting normal metric contact pairs of type $(h,0)$ are the obvious ones.
However, if one drops the normality condition, then there are plenty of other examples, of course.

Even without using the structure theorem for Vaisman manifolds, we do get very strong conclusions. Proposition~\ref{p:Tint}
tells us that Vaisman manifolds admit two integrable complex structures inducing opposite orientations. This is not interesting
in odd complex dimensions, since for those dimensions the complex conjugate of a complex structure induces the 
orientation opposite to the original one. However, in even complex dimensions, this is a rather severe restriction, which,
in complex dimension two, was first considered by Beauville~\cite{Bea}, and later by the second author~\cite{K1,Ko}.

For a given manifold $X$ we denote by $\bar X$ the same manifold endowed with the opposite orientation. 
We quote the following result proved in~\cite{Ko} and refer to~\cite{BPV} for the classification of compact complex surfaces:
\begin{theorem}[\cite{Ko}]\label{t:class-Ko}
Let X be a compact complex surface admitting a complex structure for $\bar X$. Then X (and $\bar X$) satisfies one of the following:
\begin{enumerate}
\item X is geometrically ruled, or
\item the Chern numbers $c_1^2$ and $c_2$ of $X$ vanish, or 
\item X is uniformised by the polydisk.
\end{enumerate}
In particular, the signature of X vanishes.
\end{theorem}
Note that the surfaces in the first and third cases are K\"ahler, so that non-K\"ahler Vaisman manifolds can only occur in
the second case. Further, only the first case can contain non-minimal surfaces.

\begin{corollary}[\cite{Ko}]\label{c:class-Ko}
Let X be a compact complex surface admitting a complex structure for $\bar X$ . If the Kodaira dimension of one of the two surfaces (equivalently, both of them) 
is non-negative, then $X$ carries a Thurston geometry compatible with the complex structure. The following surfaces and geometries can and do occur:
\begin{enumerate}
\item surfaces of general type with geometry $\bH \times \bH$, 
\item properly elliptic surfaces with even $b_1$ and with geometry $\bC \times \bH$,
\item properly elliptic surfaces with odd $b_1$ and with geometry $\widetilde{SL_2} \times \bR$,
\item tori and hyperelliptic surfaces with geometry $\bC^2$,
\item Kodaira surfaces with geometry $Nil^3\times \bR$ .
\end{enumerate}
\end{corollary}
Observe that a non-K\"ahler Vaisman surface is minimal, because it fibers over $S^1$, and so every embedded $2$-sphere must have 
zero self-intersection, since it is homotopic to a sphere contained in a fiber of the fibration over $S^1$.

Since a manifold endowed with a normal metric contact pair or a Vaisman structure 
carries two complex structures giving opposite orientations, 
we can apply Corollary~\ref{c:class-Ko} to obtain the following classification:
\begin{theorem}\label{t:class-normal-c-pairs}
Let $M$ be a compact complex surface endowed with a normal metric contact pair $(\alpha, \beta, J,g)$, or, equivalently, a non-K\"ahlerVaisman structure. 
If the Kodaira dimension of $M$  is non-negative then $M$ carries a Thurston geometry compatible with the complex structure and with the contact pair. 
The following surfaces and geometries can and do occur:
\begin{enumerate}
\item properly elliptic surfaces with odd $b_1$ and with geometry $\widetilde{SL_2} \times \bR$,
\item Kodaira surfaces with geometry $Nil^3\times \bR$.
\end{enumerate}
\end{theorem}
\begin{proof}
It is a result of Vaisman \cite{V} that, if a manifold endowed with a lcK structure admits a K\"ahler metric, then the lcK structure is K\"ahler, after rescaling the metric. 
By Proposition~\ref{p:equiv-normalcp-vaisman}, manifolds endowed with normal contact pairs are non-K\"ahler Vaisman manifolds. In the list of Corollary~\ref{c:class-Ko}, 
the surfaces in cases (1),(2),(4) can be excluded because they have even first Betti number and are therefore K\"ahler; see~\cite{Buch}. To prove that the remaining cases 
effectively occur, it is enough to prove the existence of a contact pair $(\alpha, \beta)$ on the geometric model, which is invariant by the isometry group and such that the 
complex structure on the model is compatible with $d\alpha -\alpha \land \beta$, intertwines the Reeb vector fields, and is preserved by the flow of one of them (cf.~\cite{BH3}). 
We discuss each case using the description given in~\cite{W1}, writing $e_1,e_2,e_3,e_4$ for a basis of the tangent space at the identity and $\omega_1,\omega_2,\omega_3,\omega_4$ 
for its dual basis.

The maximal connected isometry group of $\widetilde{SL_2} \times \bR$ is a semidirect product of $\widetilde{SL_2} \times \bR$, acting on itself by left translations, and a circle. 
The structure equations of $\widetilde{SL_2} \times \bR$ are given by
\begin{alignat*}{2}
&d\omega_1=\omega_2\wedge \omega_3 \, , \, &&d\omega_2=- \omega_1 \wedge \omega_3 \, ,\\
&d\omega_3=-\omega_1 \wedge \omega_2 \, , \, \quad \quad &&d\omega_4=0 \, .
\end{alignat*}
The circle acts by isometries, as described in \cite{W1}, with rotations on the plane $(\omega_1,\omega_2)$. The complex structure, which is compatible with the action of the circle, is defined by $Je_4=-e_3$ and $Je_1=e_2$. A left invariant contact pair is given by $(\omega_3,\omega_4)$ and its Reeb vector fields are $e_3$, $e_4$. An easy computation shows that $J$ is compatible with $d\omega_3- \omega_3 \land\omega_4$. The Reeb vector fields are intertwined by $J$ and $L_{e_4}J=0$. Thus the contact pair is normal. This contact pair is also preserved by the circle action and it descends to all quotients by cocompact lattices.

Now consider $Nil^3\times \bR$. Its maximal connected isometry group is a semidirect product of $Nil^3\times \bR$, again acting on itself by left translations, and a circle. The structure equations of $Nil^3\times \bR$ are
\begin{equation*}
d\omega_1=d\omega_2=d\omega_4=0 \, ,\, d\omega_3=-\omega_1 \wedge \omega_2 \, .
\end{equation*}
The complex structure is given by $Je_1=e_2$ and $Je_3=e_4$. Again the circle acts by rotations on the plane $(\omega_1 ,\omega_2)$. The pair $(\omega_3, \omega_4)$ gives rise to an invariant contact pair, preserved by the circle action, which descends to all quotients by cocompact lattices. The complex structure $J$ is compatible with $d\omega_3- \omega_3 \land\omega_4$. Since the Reeb vector fields $e_3$, $e_4$. are intertwined by $J$, and $L_{e_4}J=0$, the contact pair is normal.
\end{proof}

In the case of negative Kodaira dimension, the results of~\cite{Ko} do not give a classification of Vaisman manifolds. However, this classification was achieved
by Belgun~\cite{Bel}, who proved the following:
\begin{theorem}
[\cite{Bel}]
A compact complex surface of negative Kodaira dimension is Vaisman if and only if it is a Hopf surface carrying the Thurston geometry $S^3\times\R$.
\end{theorem}
The result is not formulated in this way in~\cite{Bel}, however the formulation given there is seen to be equivalent to the one above if one keeps in mind 
Wall's characterisation of manifolds with geometry $S^3\times\R$; see~\cite{W2}. Finally, note that Belgun~\cite{Bel} classified all Vaisman manifolds of complex
dimension $2$ independently of the results in~\cite{Ko}, thereby giving a different proof of Theorem~\ref{t:class-normal-c-pairs} above.

\section{Further Applications}

\subsection{Lcs versus lcK manifolds}

Since a K\"ahler structure consists of a pair of compatible complex and symplectic structures,
there are three distinct ways in which a manifold having at least one of these structures can fail to be K\"ahler:
being symplectic but not having any complex structure at all, being complex and not having a symplectic structure, 
or, most interestingly, having both, but no compatible pair. There are examples of all three kinds in the 
lowest possible dimension equal to four. Firstly, there are many symplectic four-manifolds without complex
structures, for example by the constructions of Gompf~\cite{Go}. Secondly, there are complex surfaces, such as 
Hopf surfaces, which are not even cohomologically symplectic. Finally, there are examples like the Kodaira--Thurston
manifold, which are both symplectic and complex, but cannot be K\"ahler for cohomological reasons.

In a similar vein, an lcK structure consists of a pair of compatible complex and lcs structures, and we now want
to find examples having at least one of these structures but no lcK structure.
Our first example answers the question posed by Ornea and Verbitsky as ``Open Problem 1'' in~\cite{OV2}.
\begin{example}
By our discussion in Section~\ref{s:rel}, the product of any contact manifold with the circle is lcs.
In particular, $M^3\times S^1$ is lcs for any orientable $3$-manifold $M$, since any such $M$ is contact by
classical results of Lutz and Martinet. Note that only for very special $M$ (those which fiber over the circle),
can $M^3\times S^1$ be genuinely symplectic~\cite{FV}. Moreover, for most choices of $M$, this product has
no complex structure. For example, if $M$ is hyperbolic, then $M^3\times S^1$ cannot be complex by~\cite[Example 3.7]{KK}.
Thus there exist lcs manifolds without complex (or symplectic) structures. In particular, these manifolds are not lcK.
\end{example}

We cannot give such examples in higher dimensions, because every lcs manifold is almost complex, and in higher 
dimensions there are no known obstructions for almost complex manifolds to admit complex structures.

\begin{example}
Consider $S^{2n+1}\times S^{2l+1}$ with $n, l \geq 1$. These manifolds carry the Calabi--Eckmann
complex structures, but they cannot be lcs for cohomological reasons. Their first Betti numbers vanish, so the Lee form 
of any lcs structure would be exact, and so an lcs structure would be globally conformally symplectic. However, 
since the second Betti numbers also vanish, these manifolds are not symplectic.
\end{example}
The real dimensions of these examples are $\geq 6$, and it is very likely that no such examples exist in dimension $4$:
\begin{conjecture}\label{conj}
Any closed smooth four-manifold that admits a complex structure also admits an lcK structure (not necessarily
compatible with the given complex structure).
\end{conjecture}
The stronger statement that all complex structures on complex surfaces should have compatible lcK metrics is false by an example 
due to Belgun~\cite{Bel}. He found that certain Inoue surfaces admit no lcK structure compatible with the given 
complex structure. However, after deforming the complex structure, these surfaces do become lcK, and so they are 
not counterexamples to our weaker conjecture. The conjecture holds for all known examples of complex surfaces,
by combining the results of Belgun~\cite{Bel} (and work of other authors mentioned in~\cite{Bel}) and Brunella~\cite{Bru}, 
since by these results all known compact
complex surfaces are lcK, except for the Inoue surfaces, where the statement is true only up to deformation. The standard 
conjectures~\cite{Nak,Tel} about the classification of surfaces of class VII imply our conjecture.

\begin{remark}
Brunella~\cite{Bru} raised the question whether the universal covering of every compact complex surface is K\"ahler.
This is true for all known surfaces by~\cite{Bel,Bru}, since the only non-lcK ones are certain Inoue surfaces 
with universal covering $\bC\times\bH$. The standard conjectures~\cite{Nak,Tel} about the classification of surfaces of 
class VII imply a positive answer to Brunella's question.
\end{remark}

Finally we want to give examples of manifolds which are complex and lcs, but are not lcK. Recall that the 
Kodaira--Thurston manifold is complex and symplectic, but not K\"ahler. However, it is lcK.
\begin{example}
Let $\Gamma$ be a finitely presentable group with $b_1(\Gamma)=0$ which is not the fundamental group of any compact
K\"ahler manifold. Such groups exist, for example one may take the fundamental group of a hyperbolic homology sphere,
cf.~\cite{ABCKT}. By a result of Gompf~\cite{Go} one can find a closed symplectic $4$-manifold $X$ with $\pi_1(X)=\Gamma$.
For any $k>0$, the $k$-fold blowup $X_k$ of $X$ is still symplectic with the same fundamental group. As soon as $k$ is 
large enough, the twistor space $Z$ of $X_k$ admits a complex structure by a result of Taubes~\cite{Taubes}. The 
twistor space is a $2$-sphere bundle over $X_k$, and so has the same fundamental group $\Gamma$. Moreover, by 
the Thurston construction, the twistor space is symplectic since $X_k$ is. Thus $Z$ has complex and symplectic structures.
However, it cannot be lcK. The reason is that the vanishing of the first Betti number implies that any lcK structure would be 
globally conformally K\"ahler, contradicting the assumption made about $\Gamma$.
\end{example}
It would be interesting to have such examples which are complex and lcs without being genuinely symplectic. Such 
examples are difficult to pin down because Conjecture~\ref{conj} implies that one cannot find complex non-lcK four-manifolds, 
and in higher dimensions we have no arguments yet to rule out the existence of lcK structures on complex manifolds, other than 
the one used above, which reduces to the K\"ahler case using the assumption $b_1=0$. (Compare the Postscript to~\cite{KK}.)
Of course this assumption makes every lcs structure globally conformally symplectic. Nevertheless,
there is an important difference between the generalizations from symplectic to lcs on the one hand, and from K\"ahler to lcK
on the other. Whereas symplectic and non-globally conformally symplectic lcs structures can both exist on the 
same smooth manifold (for example the Kodaira--Thurston manifold), K\"ahler and non-globally conformally K\"ahler lcK structures
never exist on the same complex manifold by a result of Vaisman~\cite{V}. 

\subsection{Normal contact pairs versus K\"ahler pairs}

In \cite{BK} we introduced the notion of symplectic pairs on even dimensional manifolds. We recall the definition in the four dimensional case:
\begin{definition}
A symplectic pair on a $4$-dimensional manifold is a pair of closed two-forms $(\omega_1, \omega_2)$ such that $\omega_1 \land\omega_2$ is a volume form and $(\omega_1)^2=(\omega_2)^2=0$.
\end{definition}
In particular $\omega_1 \pm \omega_2$ are symplectic forms and give opposite orientations.

One can translate to symplectic pairs the condition of normality given for contact pairs. This goes as follows: the tangent bundle of a manifold endowed with a symplectic pair splits into a direct sum of rank $2$ symplectic bundles and each of them can be endowed with an almost complex structure, say $J_1$ and $J_2$ and compatible metrics $g_1$, $g_2$. Then we obtain two almost complex structures $J=J_1\oplus J_2$ and $T=J_1\oplus (-J_2)$. The metric $g=g_1 \oplus g_2$ is compatible with both $J$ and $T$ and its fundamental forms are $\omega_1 +\omega_2$ and $\omega_1 -\omega_2$ respectively. We say that the symplectic pair is normal if both $J$ and $T$ are integrable. This in turn implies that $\omega_1 +\omega_2$ and $\omega_1 -\omega_2$ are K\"ahler. 
This structure has been studied in \cite{G} and is called {\it K\"ahler pair}.

As in the case of contact pairs, $J$ and $T$ give opposite orientations. Then we can apply Corollary \ref{c:class-Ko} and obtain:
\begin{theorem}\label{t:class-K-pairs-posK}
Let $M$ be a compact complex surface endowed with a K\"ahler pair. If the Kodaira dimension  is non-negative then $M$ is one of the following:
\begin{enumerate}
\item surfaces of general type with geometry $\bH \times \bH$, 
\item properly elliptic surfaces with even $b_1$ and geometry $\bC \times \bH$,
\item tori and hyperelliptic surfaces ith geometry $\bC \times \bC$.
\end{enumerate}
\end{theorem}
\begin{proof}
Since a K\"ahler pair gives rise to two K\"ahler forms, the only possible cases are those listed in the theorem.

It is clear that each model carries a K\"ahler pair. Indeed the symplectic pair is given by the obvious K\"ahler forms on the factors and the  complex structures  are $J=J_1 \oplus J_2$ and $T=J_1\oplus (-J_2)$, 
where $ J_1$, $J_2$ are the complex structures on the factors. This implies for example, that the hyperelliptic surfaces carry a K\"ahler pair compatible with the geometry, because they are all quotients of $\bC \times \bC$ by isometries which preserve the local product structure and the orientation on each factor. Nevertheless, in general the K\"ahler pairs on the geometric model are not invariant under the action of the full maximal connected isometry group, but there are smaller subgroups which preserve them (cf. \cite[Examples 8,9,10]{BK}). 
\end{proof}

\bigskip

\bibliographystyle{amsalpha}

\end{document}